\documentclass{amsart}
\usepackage{amsfonts,amscd,amsthm,amsgen,amsmath,amssymb}
\usepackage[all]{xy}
\usepackage{epsfig,color}
\usepackage[vcentermath]{youngtab}
\newtheorem{theorem}{Theorem}[section]

\newtheorem{corollary}[theorem]{Corollary}
\newtheorem{proposition}[theorem]{Proposition}

\theoremstyle{definition}
\newtheorem{definition}[theorem]{Definition}

\theoremstyle{remark}
\newtheorem{remark}[theorem]{Remark}

\numberwithin{equation}{section}

\begin{document}

\setlength\parskip{0.5em plus 0.1em minus 0.2em}

\title{On the mapping class groups of simply-connected smooth $4$-manifolds}
\author{David Baraglia}

\address{School of Mathematical Sciences, The University of Adelaide, Adelaide SA 5005, Australia}

\email{david.baraglia@adelaide.edu.au}

%\subjclass[2010]{Primary 53C08, 19L50; Secondary 53D18, 53C80}

\date{\today}

\begin{abstract}

The mapping class group $M(X)$ of a smooth manifold $X$ is the group of smooth isotopy classes of orientation preserving diffeomorphisms of $X$. We prove a number of results about the mapping class groups of compact, simply-connected, smooth $4$-manifolds. We prove that $M(X)$ is non-finitely generated for $X = 2n \mathbb{CP}^2 \# 10n \overline{\mathbb{CP}^2}$, where $n \ge 3$ is odd. Let $\Gamma(X)$ denote the group of automorphisms of the intersection lattice of $X$ that can be realised by diffeomorphisms. Then $M(X)$ is an extension of $\Gamma(X)$ by $T(X)$, the Torelli group of isotopy classes of diffeomorphisms that act trivially in cohomology. We prove that this extension is split for connected sums of $\mathbb{CP}^2$, but is not split for $2\mathbb{CP}^2 \# n \overline{\mathbb{CP}^2}$, where $n \ge 11$. We prove that the Nielsen realisation problem fails for certain finite subgroups of $M( p \mathbb{CP}^2 \# q \overline{\mathbb{CP}^2} )$ whenever $p+q \ge 4$. Lastly we study the extension $M_1(X) \to M(X)$, where $M_1(X)$ is the group of isotopy classes of diffeomorphisms of $X$ which fix a neighbourhood of a point. When $X = K3$ or $K3 \# (S^2 \times S^2)$ we prove that $M_1(X) \to M(X)$ is a non-trivial extension of $M(X)$ by $\mathbb{Z}_2$. Moreover, we completely determine the extension class of $M_1(K3) \to M(K3)$.

\end{abstract}

%%%%%%%%%%%%%%%%%%%%%%%%%%%%%%%%%%%%%%%%%%%%%%%%%%%%%%%%%%%%%%%%%%%%%%%%%%%%%%%%
%%%%%%%%%%%%%%%%%%%%%%%%%%%%%%%%%%%%%%%%%%%%%%%%%%%%%%%%%%%%%%%%%%%%%%%%%%%%%%%%
%%%%%%%%%%%%%%%%%%%%%%%%%%%%%%%%%%%%%%%%%%%%%%%%%%%%%%%%%%%%%%%%%%%%%%%%%%%%%%%
%%%%%%%%%%%%%%%%%%%%%%%%%%%%%%%%%%%%%%%%%%%%%%%%%%%%%%%%%%%%%%%%%%%%%%%%%%%%%%%%

\maketitle

%%%%%%%%%%%%%%%%%%%%%%%%%%%%%%%%%%%%%%%%%%%%%%%%%%%%%%%%%%%%%%%%%%%%%%%%%%%%%%%%
%%%%%%%%%%%%%%%%%%%%%%%%%%%%%%%%%%%%%%%%%%%%%%%%%%%%%%%%%%%%%%%%%%%%%%%%%%%%%%%%
%%%%%%%%%%%%%%%%%%%%%%%%%%%%%%%%%%%%%%%%%%%%%%%%%%%%%%%%%%%%%%%%%%%%%%%%%%%%%%%%
%%%%%%%%%%%%%%%%%%%%%%%%%%%%%%%%%%%%%%%%%%%%%%%%%%%%%%%%%%%%%%%%%%%%%%%%%%%%%%%%

%%%%%%%%%%%%%%%%%%%%%%%%%%%%%%%%%%%%%%%%%%

%%%%%%%%%%%%%%%%%%%%%%%%%%%%%%%

\section{Introduction}

Let $X$ be a compact, oriented, smooth, simply-connected $4$-manifold. Define the mapping class group $M(X)$ to be the group of smooth isotopy classes of orientation preserving diffeomorphisms of $X$. There is considerable interest in the groups $M(X)$, although little is known about their structure. In this paper we will prove a number of new results concerning the structure of mapping class groups of smooth $4$-manifolds.

Recall that the second cohomology group $L_X = H^2(X ; \mathbb{Z})$ of $X$ equipped with its intersection form is a unimodular lattice. We let $Aut(L_X)$ denote the automorphism group of the lattice $L_X$. The group of orientation preserving diffeomorphisms of $X$ acts on $L_X$ via $f \mapsto (f^{-1})^*$. This action depends only on this isotopy class and so defines a homomorphism $M(X) \to Aut(L_X)$. Denoting the image of this map by $\Gamma(X)$ and the kernel by $T(X)$, we obtain a short exact sequence
\begin{equation}\label{equ:ext}
1 \to T(X) \to M(X) \to \Gamma(X) \to 1.
\end{equation}
We call $T(X)$ the {\em Torelli group} of $X$. It is the group of isotopy classes of diffeomorphisms of $X$ that act trivially in cohomology. By a result of Quinn, $T(X)$ can also be defined as the group of isotopy classes of diffeomorphisms which are continuously isotopic to the identity \cite{qui}. The group $\Gamma(X)$ is the group of automorphisms of $L_X$ that can be realised by diffeomorphisms of $X$.

Understanding the group $M(X)$ necessitates an understanding of the groups $T(X)$, $\Gamma(X)$ and the extension (\ref{equ:ext}). The group $\Gamma(X)$ is known for some classes of $4$-manifolds. In particular, a theorem of Wall implies that $\Gamma(X) = Aut(L_X)$ for a large class of $4$-manifolds \cite{wall}. In contrast, the Torelli group $T(X)$ is poorly understood. Ruberman showed that $T(X)$ is not finitely generated for certain $X$ \cite{rub3}. However this does not imply that $M(X)$ is not finitely generated, since a finitely generated group can have subgroups which are not finitely generated. Our first main result confirms that $M(X)$ is not finitely generated for certain simply-connected $4$-manifolds.

\begin{theorem}\label{thm:nfg0}
Let $X = 2n \mathbb{CP}^2 \# 10n \overline{\mathbb{CP}^2}$, where $n \ge 3$ is odd. Then $M(X)$ is not finitely generated. More precisely, the following holds:
\begin{itemize}
\item[(1)]{There is an index $2$ subgroup $M_+(X)$ of $M(X)$ and a surjective homomorphism $\Phi : M_+(X) \to \mathbb{Z}^\infty$ from $M_+(X)$ to $\mathbb{Z}^{\infty}$, where $\mathbb{Z}^{\infty}$ denotes a free abelian group of countably infinite rank.}
\item[(2)]{The mod $2$ reduction of $\Phi$ extends to a surjective homomorphism $\Phi : M(X) \to \mathbb{Z}_2^\infty$.}
\end{itemize}
\end{theorem}

As this paper was nearing completion we received a preprint by Hokuto Konno \cite{kon} which also proves that the mapping class groups of simply-connected $4$-manifolds can be non-finitely generated. Konno's proof uses essentially the same method as ours, however we obtained our proofs completely independently.

\begin{remark}
It is interesting to contrast Theorem \ref{thm:nfg0} with finiteness results for mapping class groups in other dimensions. Let $X$ be a compact, simply-connected smooth manifold of dimension $d$ and $M(X) = \pi_0(Diff(X))$ the mapping class group. If $d \neq 4$, then $M(X)$ is finitely generated. For $d \le 3$, finite generation holds for any compact oriented manifold (see \cite{dehn} for $d=2$ and \cite{hm} for $d=3$). If $d \ge 5$ then $M(X)$ is finitely generated \cite[Theorem 2.6]{bkk}. Note that Theorem 2.6 of \cite{bkk} is only stated for $d \ge 6$, but when $X$ is simply-connected, the proof carries over to $d=5$. In the proof of \cite[Theorem 2.6]{bkk}, dimension $6$ only enters in the point (i) in the proof, but Cerf's theorem says that in the simply-connected case $\pi_0(C^{Diff}(X)) = 0$, and in \cite[Proposition 2.7]{bkk} where it is not necessary. This follows from specialising Triantafillou \cite{tri} to simply-connected manifolds, where none of the oversights mentioned in \cite[\textsection 2.2]{bkk} cause a problem\footnote{I thank Alexander Kupers for explaining why \cite[Theorem 2.6]{bkk} works for simply-connected $5$-manifolds.}.

One may also consider the larger group $M'(X)$ consisting is isotopy classes of diffeomorphisms which are not necessarily orientation preserving. Since $M(X)$ has finite index in $M'(X)$, it follows from Schreier's lemma \cite{ser} that if $M(X)$ is not finitely generated then neither is $M'(X)$.
\end{remark}

\begin{remark}
Let $X$ be a compact, simply-connected smooth $4$-manifold and let $M^{top}(X) = \pi_0( Homeo(X) )$ be the topological mapping class group. By work of Freedman \cite{fre} and Quinn \cite{qui}, the natural map $M^{top}(X) \to Aut( H^2(X ; \mathbb{Z}))$ to the group of automorphisms of the intersection lattice $H^2(X ; \mathbb{Z})$ is an isomorphism. By a result of Siegel \cite{sie}, the automorphism group of any lattice is finitely generated. Hence $M^{top}(X)$ is finitely generated.

In contrast, we do not know whether the group $\Gamma(X) \subseteq Aut(H^2(X ; \mathbb{Z}))$ is always finitely generated, although we conjecture that it is.
\end{remark}

Our next result concerns the question of whether or not the sequence (\ref{equ:ext}) admits a splitting.

\begin{theorem}\label{thm:sns}
The following holds:
\begin{itemize}
\item[(1)]{Let $X = n \mathbb{CP}^2$, where $n \ge 1$. Then there exists a splitting $\Gamma(X) \to M(X)$.}
\item[(2)]{Let $X = (S^2 \times S^2) \# X'$, where $b_+(X') = 1$, $b_-(X') \ge 10$. Then there does not exist a splitting $\Gamma(X) \to M(X)$.}
\end{itemize}
\end{theorem}

More precise information about the failure of a splitting in case (2) is provided by Theorem \ref{thm:nolift}.

\begin{remark}
In \cite{bk3} it is shown when $X$ is a $K3$ surface, there is a splitting $\Gamma(X) \to M(X)$. It is also easy to see that splittings exist for $S^2 \times S^2$ and $\mathbb{CP}^2 \# \overline{\mathbb{CP}^2}$.
\end{remark}

Our next result concerns the Nielsen realisation problem. Recall that the Nielsen realisation problem for a smooth manifold $X$ asks whether a subgroup $G$ of the mapping class group of $X$ can be lifted to a subgroup of $Diff(X)$. Recent results of Baraglia--Konno \cite{bk3}, Farb--Looijenga \cite{fl}, Konno \cite{kon0} show that Nielsen realisation fails for many simply-connected spin $4$-manifolds. Arabadji--Baykur showed that there are many non-spin $4$-manifolds with finite non-trivial fundamental group for which Nielsen realisation fails \cite{ab} and Konno--Miyazawa--Taniguchi gave examples with simply-connected indefinite non-spin $4$-manifolds \cite{kmt0}.

\begin{theorem}\label{thm:n}
Let $X = X' \# p \mathbb{CP}^2 \# q \overline{\mathbb{CP}^2}$ where $X'$ is a compact, smooth, simply-connected $4$-manifold and $p+q \ge 4$. Then $M(X)$ contains a subgroup isomorphic to $\mathbb{Z}_2^4$ which can not be lifted to $Diff(X)$.
\end{theorem}

In particular, Nielsen realisation fails for $n \mathbb{CP}^2$ for $n \ge 4$. As far as we are aware, these are the first examples of definite, simply-connected $4$-manifolds where Nielsen realisation fails.

Our last main result concerns a certain extension of $M(X)$. Let $X^{(1)}$ be obtained from $X$ by removing an open ball and let $Diff(X^{(1)} , \partial X^{(1)})$ denote the group of diffeomorphisms of $X^{(1)}$ which are the identity in a neighbourhood of the boundary. Let $M_1(X) = \pi_0( Diff( X^{(1)} , \partial X^{(1)} ))$ denote the group of components of $Diff( X^{(1)} , \partial X^{(1)})$. It is known that the map $M_1(X) \to M(X)$ is surjective and that the kernel (which is either trivial or has order $2$) is generated by a Dehn twist on the boundary (see Section \ref{sec:bdt} for more details).

In general it is difficult to determine whether the kernel of $M_1(X) \to M(X)$ is trivial or non-trivial, or equivalently, whether the boundary Dehn twist is trivial or non-trivial. The extension is known to be trivial when $X$ is a connected sum of copies of $S^2 \times S^2$. In contrast we have:

\begin{theorem}\label{thm:phi0}
Let $X'$ be a compact, smooth, simply-connected $4$-manifold which is homeomorphic to $K3$. Let $X$ be $X'$ or $X' \# (S^2 \times S^2)$. Then the boundary Dehn twist is non-trivial. Moreover, the extension $1 \to \mathbb{Z}_2 \to M_1(X) \to M(X) \to 1$ does not split.
\end{theorem}

If $M_1(X) \to M(X)$ is a non-trivial extension, then it is given by an extension class $\xi_X \in H^2( M(X) ; \mathbb{Z}_2)$ and the above theorem says that $\xi_X \neq 0$ when $X$ is of the stated form. Our final result completely determines $\xi_X$ in the case that $X$ is homeomorphic to $K3$. Let $L_X$ be the intersection lattice of $X$ and $Aut(L_X)$ the group of automorphisms. Over the classifying space $BAut(L_X)$ we have the tautological flat bundle $H = EAut(L_X) \times_{Aut(L_X)} L_X$. Let $H^+ \to BAut(L_X)$ be a maximal positive subbundle. This defines a characteristic class $w_2(H^+) \in H^2( Aut(L_X) ; \mathbb{Z}_2)$.

\begin{theorem}\label{thm:pull}
Let $X$ be a smooth $4$-manifold which is homeomorphic to $K3$. Then the extension class $\xi_X \in H^2( M(X) ; \mathbb{Z}_2)$ is the pullback of $w_2(H^+) \in H^2( Aut(L_X) ; \mathbb{Z}_2 )$ under the map $M(X) \to Aut(L_X)$.
\end{theorem}

%%%%%%%%%%%%%%%%%%%%%%%%%%%%%%%%%%%%%%%%%
\subsection{Structure of the paper}

The structure of the paper is as follows. In Section \ref{sec:sw1} we review the Seiberg--Witten invariants for the Torelli group (as in \cite{rub1,rub2,bk1}) and show how these invariants can be assembled into cohomology classes on the mapping class group. In Section \ref{sec:sw2} we use these cohomology classes to show that $M(X)$ is not finitely generated for certain $X$. In Section \ref{sec:split} we construct a splitting $\Gamma(X) \to M(X)$ when $X$ is a connected sum of copies of $\mathbb{CP}^2$. In Section \ref{sec:nsplit} we prove the non-existence of splittings $\Gamma(X) \to M(X)$ for certain $4$-manifolds. The proof uses families Seiberg--Witten theory and more specifically the main result of \cite{bar1}. In Section \ref{sec:nielsen} we prove Theorem \ref{thm:n}. Finally, in Section \ref{sec:bdt} we study boundary Dehn twists and the extension $M_1(X) \to M(X)$ and we prove Theorems \ref{thm:phi0} and \ref{thm:pull}.

%%%%%%%%%%%%%%%%%%%%%%%%%%%%%%%%%%%%%%%%%%%%%%
\subsection{Acknowledgements}

I thank Hokuto Konno for sending me his paper \cite{kon} and also for comments on a draft of this paper. I also thank Alexander Kupers for clarifying the finite generation of mapping class groups for simply-connected $5$-manifolds.

%%%%%%%%%%%%%%%%%%%%%%%%%%%%%%%%%%%%%%%%%%%%%%%%
\section{Seiberg--Witten invariants for the mapping class group}\label{sec:sw1}

In this section we define Seiberg--Witten invariants for the mapping class group, extending the Seiberg--Witten invariants on the Torelli group which have previously been considered in \cite{rub1,rub2,bk1}. These invariants will be used to show that certain simply-connected $4$-manifolds have non-finitely generated mapping class group.

Let $X$ be a compact, smooth, simply connected $4$-manifold and let $\mathfrak{s}$ be a spin$^c$-structure with $d(\mathfrak{s}) = -1$, where
\[
d(\mathfrak{s}) = \frac{ c(\mathfrak{s})^2 - \sigma(X) }{4} - b_+(X) - 1
\]
is the expected dimension of the Seiberg--Witten moduli space for $\mathfrak{s}$. Let $\mathcal{S}(X)$ denote the set of all isomorphism classes of spin$^c$-structures on $X$ for which $d(\mathfrak{s}) = -1$. Since $X$ is assumed to be simply connected, $\mathcal{S}(X)$ can be identified with the set of characteristic elements $c \in L = H^2(X ; \mathbb{Z})$ for which $( c^2 - \sigma(X))/4 - b_+(X) = 0$.

Let $\Pi$ denote the space of pairs $(g , \eta)$ where $g$ is a Riemannian metric on $X$ and $\eta$ is a $2$-form which is self-dual with respect to $g$. For any $h \in \Pi$ and any $\mathfrak{s} \in \mathcal{S}(X)$ we may consider the Seiberg--Witten equations on $X$ with respect to the metric $g$, spin$^c$-structure $\mathfrak{s}$ and $2$-form perturbation $\eta$. Let $\mathcal{M}(X,\mathfrak{s},h)$ denote the moduli space of gauge equivalence classes of solutions to the Seiberg--Witten equations for $(X,\mathfrak{s},h)$. Assume $b_+(X) > 2$. We will say that $h \in \Pi$ is regular if $\mathcal{M}(X , \mathfrak{s},h)$ is empty for all $\mathfrak{s} \in S(X)$. Since $b_+(X) > 0$ and the expected dimension of $\mathcal{M}(X,\mathfrak{s},h)$ is negative, the regular elements form a subset of $\Pi$ of Baire second category with respect to the $\mathcal{C}^\infty$ topology. Let $\Pi^{reg} \subseteq \Pi$ denote the set of regular elements.

Suppose that $h_0, h_1 \in \Pi^{reg}$. If $h : [0,1] \to \Pi$ is a path in $\Pi$ from $h_0$ to $h_1$, we can consider the families moduli space, which is the union over $t \in [0,1]$ of the Seiberg--Witten moduli spaces for each $h_t \in \Pi$. For a sufficiently generic path $h_t$, the moduli space is a compact, smooth, $0$-dimensional manifold. A choice of orientation on a maximal positive definite subspace of $H^2(X ; \mathbb{R})$ determines an orientation on the moduli space and hence we can count with sign the number of points in the moduli space. Fix a choice of such an orientation.  It can be shown \cite{rub1} that the number of solutions depends on the endpoints $h_0, h_1$, but not on the choice of generic path $h_t$. Hence we may denote by $SW_{\mathfrak{s}}( h_0 , h_1) \in \mathbb{Z}$ the signed count of points in the moduli space. From the definition it is clear that this count of points satisfies the following properties:
\begin{itemize}
\item[(1)]{$SW_{\mathfrak{s}}( h_0 , h_1) + SW_{\mathfrak{s}}(h_1 , h_2) = SW_{\mathfrak{s}}(h_0 , h_2)$,}
\item[(2)]{$SW_{\mathfrak{s}}(h_0 , h_1 ) = sgn_+(f) SW_{f(\mathfrak{s})}( f(h_0) , f(h_1) )$ for any orientation preserving diffeomorphism $f$.}
\end{itemize}

In (2), $sgn_+(f)$ is defined as follows. The space of oriented, maximal positive definite subspaces of $H^2(X ; \mathbb{R})$ has two connected components. For an isometry $\varphi$ of $H^2(X ; \mathbb{R})$ we let $sgn_+(\varphi)$ equal $1$ or $-1$ according to whether $\varphi$ preserves or exchanges the two components. If $f$ is an orientation preserving diffeomorphism of $X$, then $sgn_+(f)$ denotes $sgn_+(f_*)$, where $f_* = (f^{-1})^*$ is the isometry of $H^2(X ; \mathbb{R})$ induced by $f$.

Property (1) follows by concatenating a path from $h_0$ to $h_1$ with a path from $h_1$ to $h_2$. Property (2) follows from diffeomorphism invariance of the Seiberg--Witten equations. In addition, $SW_{\mathfrak{s}}(h_0,h_1)$ obeys a symmetry with respect to charge conjugation:
\begin{itemize}
\item[(3)]{$SW_{\mathfrak{s}}(h_0 , h_1) = (-1)^{b_+(X)/2+1} SW_{\bar{\mathfrak{s}} }( \bar{h}_0 , \bar{h}_1)$,}
\end{itemize}
where $\bar{\mathfrak{s}}$ denotes the charge conjugate of $\mathfrak{s}$ and for $h = (g,\eta) \in \Pi$, we set $\bar{h} = (g , -\eta)$. Property (3) is an immediate consequence of the charge conjugation symmetry of the Seiberg--Witten equations.

Let $f \in T(X)$ be an element of the Torelli group. Fix a spin$^c$-structure $\mathfrak{s}$ with $d(\mathfrak{s}) = -1$. The mapping cylinder of $f$ defines a smooth family $E \to S^1$ over $S^1$ with fibres diffeomorphic to $X$. Since $f$ acts trivially on cohomology, it preserves the isomorphism class of $\mathfrak{s}$. It follows easily that there is a unique spin$^c$-structure on the vertical tangent bundle of $E$ which restricts to $\mathfrak{s}$ on each fibre. Since $b_+(X) > 2$, there is a single chamber for the families Seiberg--Witten equations for $E$. Furthermore, the families moduli space is oriented and so we obtain an integer-valued invariant $SW_{\mathfrak{s}}(f) \in \mathbb{Z}$ which depends only on $(X , \mathfrak{s})$ and the isotopy class of $f$ (see \cite{bk1} for more details). From the definition of $SW_{\mathfrak{s}}(f)$, it is easy to see that
\[
SW_{\mathfrak{s}}(f) = SW_{\mathfrak{s}}( h , f(h) )
\]
for any $h \in \Pi^{reg}$. It is instructive to see why $SW_{\mathfrak{s}}(f)$ is independent of the choice of $h \in \Pi^{reg}$. Let $h' \in \Pi^{reg}$. Then
\begin{align*}
SW_{\mathfrak{s}}( h' , f(h') ) &= -SW_{\mathfrak{s}}( h , h') + SW_{\mathfrak{s}}(h , f(h) ) + SW_{\mathfrak{s}}( f(h) , f(h') ) \\
&= -SW_{\mathfrak{s}}(h,h') + SW_{\mathfrak{s}}(h,f(h)) + SW_{f^{-1}(\mathfrak{s})}( h , h') \\
&= SW_{\mathfrak{s}}(h , f(h)),
\end{align*}
where the last line follows from $f^{-1}(\mathfrak{s}) = \mathfrak{s}$, which holds since $f \in T(X)$.

For $f,g \in T(X)$, we have that $SW_{\mathfrak{s}}(f \circ g ) = SW_{\mathfrak{s}}(f) + SW_{\mathfrak{s}}(g)$ (this is a special case of Proposition \ref{prop:cocyc}, proven below). Therefore $SW_{\mathfrak{s}}$ defines a homomorphism
\[
SW_{\mathfrak{s}} : T(X) \to \mathbb{Z}
\]
or equivalently, a cohomology class $SW_{\mathfrak{s}} \in H^1( T(X) ; \mathbb{Z})$. These cohomology classes generally do not extend to the full mapping class group $M(X)$, because $\Gamma(X)$ acts non-trivially on the set of spin$^c$-structures.

Recall that the compactness of the Seiberg--Witten moduli space follows from a priori bounds. These bounds depend on the pair $h \in \Pi$, but not on the spin$^c$-structure. This argument also works for families over a compact base space, hence for fixed $f \in T(X)$, $SW_{\mathfrak{s}}(f)$ is non-zero for only finitely many $\mathfrak{s} \in S(X)$. Therefore, we can collect the homomorphisms $SW_{\mathfrak{s}}$ into a single invariant
\begin{align*}
SW : T(X) &\to \bigoplus_{\mathfrak{s} \in S(X)} \mathbb{Z} \\ f &\mapsto \bigoplus_{\mathfrak{s}} SW_{\mathfrak{s}}(f).
\end{align*}

In what follows, we will see that $SW$ can be extended from $T(X)$ to the full mapping class group $M(X)$ as a cohomology class valued in a certain $\Gamma(X)$-module.

Recall that each $\mathfrak{s} \in S(X)$ is determined by the corresponding characteristic element $c(\mathfrak{s}) \in L$. Therefore the group $\Gamma(X)$ acts on $\mathcal{S}(X)$ and hence on $\mathbb{Z}[\mathcal{S}(X)]$, the free abelian group with basis $\mathcal{S}(X)$. Let $\widehat{\mathbb{Z}}$ denote $\mathbb{Z}$ equipped with the action of $\Gamma(X)$ such that $f \in \Gamma(X)$ acts as multiplication by $sgn_+(f)$. Let $\widehat{\mathbb{Z}}[\mathcal{S}(X)] = \widehat{\mathbb{Z}} \otimes_{\mathbb{Z}} \mathbb{Z}[\mathcal{S}(X)]$. It will be convenient to regard $\widehat{\mathbb{Z}}[\mathcal{S}(X)]$ as the group of functions $\phi : \mathcal{S}(X) \to \mathbb{Z}$ with finite support. Then the action of $f \in \Gamma(X)$ is given by $(f\phi)(\mathfrak{s}) = sgn_+(f) \phi( f^{-1}(\mathfrak{s}))$. We will show that the families Seiberg--Witten invariant for $1$-dimensional families (where $b_+(X) > 2$) can be viewed as an element of $H^1( M(X) ; \widehat{\mathbb{Z}}[\mathcal{S}(X)])$. 

Recall that for a group $G$ and a $G$-module $M$, the group $H^1( G ; M )$ can be viewed as the set of equivalence classes of twisted homomomorphisms $G \to M$. A twisted homomorphism is a map $\phi : G \to M$ such that $\phi(gh) = \phi(g) + g \phi(h)$. A trivial twisted homomorphism is a twisted homomorphism of the form $\phi(g) = gm - m$ for some $m \in M$. Two twisted homomorphisms are considered equivalent if they differ by a trivial twisted homomorphism.

Let $h \in \Pi^{reg}$. Define a map $\phi_h : Diff_+(X) \to \widehat{\mathbb{Z}}[\mathcal{S}(X)]$ from the group of orientation preserving diffeomorphisms to $\widehat{\mathbb{Z}}[\mathcal{S}(X)]$ by setting
\[
(\phi_h(f))(\mathfrak{s}) = SW_{\mathfrak{s}}( h , f(h)).
\]

Suppose that $f_0,f_1 \in Diff_+(X)$ are isotopic. Choose an isotopy $f_t$. Then 
\begin{align*}
\phi_h(f_1) &= SW_{\mathfrak{s}}( h , f_1(h))\\
& = SW_{\mathfrak{s}}(h , f_0(h)) + SW_{\mathfrak{s}}( f_0(h) , f_1(h) ) \\
&= \phi_h(f_0) + SW_{\mathfrak{s}}( f_0(h) , f_1(h) ).
\end{align*}

Consider the path $h_t = f_t(h)$ from $f_0(h)$ to $f_1(h)$. By diffeomorphism invariance of the Seiberg--Witten equations, the Seiberg--Witten moduli space for $h_t$ is empty for each $t \in [0,1]$, hence $SW_{\mathfrak{s}}( f_0(h) , f_1(h) ) = 0$ and $\phi_h(f_1) = \phi_h(f_2)$. This shows that $\phi$ only depends on the underlying isotopy class and so we may view it as a map $\phi_h : M(X) \to \widehat{\mathbb{Z}}[\mathcal{S}(X)]$.

\begin{proposition}\label{prop:cocyc}
The map $\phi_h : M(X) \to \widehat{\mathbb{Z}}[\mathcal{S}(X)]$ is a twisted homomorphism. Furthermore, the underlying cohomology class $[\phi_h] \in H^1( M(X) ; \widehat{\mathbb{Z}}[\mathcal{S}(X)])$ does not depend on the choice of $h \in \Pi^{reg}$.
\end{proposition}
\begin{proof}
Let $f,g \in M(X)$. Then
\begin{align*}
\phi_h(gf)(\mathfrak{s}) &= SW_{\mathfrak{s}}( h , gf(h) ) \\
&= SW_{\mathfrak{s}}(h , g(h) ) + SW_{\mathfrak{s}}( g(h) , g(f(h)) ) \\
&= SW_{\mathfrak{s}}(h , g(h) ) + sgn_+(g) SW_{ g^{-1}\mathfrak{s}}( h , f(h) ) \\
&= \phi_h(g)(\mathfrak{s}) + sgn_+(g)\phi_h(f)( g^{-1}\mathfrak{s} ) \\
&= ( \phi_h(g) + (g\phi_h)(f) )(\mathfrak{s}).
\end{align*}
Hence $\phi_h$ is a twisted homomorphism. Next we show that the underlying cohomology class of $\phi_h$ does not depend on the choice of $h$. Let $h' \in \Pi^{reg}$ be another generic pair. Choose a path $h_t$ from $h = h_0$ to $h' = h_1$. Then
\begin{align*}
\phi_{h'}(f)(\mathfrak{s}) &= SW_{\mathfrak{s}}( h_1 , f(h_1) ) \\
&= SW_{\mathfrak{s}}( h_0 , f(h_0) ) - SW_{\mathfrak{s}}(h_0 , h_1) + SW_{\mathfrak{s}}( f(h_0) , f(h_1) ) \\
&= \phi_h(f)(\mathfrak{s}) + sgn_+(f) SW_{f^{-1}\mathfrak{s}}( h_0 , h_1 ) - SW_{\mathfrak{s}}(h_0 , h_1) \\
&= \phi_h(f)(\mathfrak{s}) + (fm-m)(\mathfrak{s})
\end{align*}
where $m(\mathfrak{s}) = SW_{\mathfrak{s}}( h_0 , h_1 )$. Hence $\phi_h$ and $\phi_{h'}$ define the same cohomology class.
\end{proof}

\begin{definition}
We denote by
\[
SW \in H^1( M(X) ; \widehat{\mathbb{Z}}[\mathcal{S}(X)] )
\]
the cohomology class $SW = [\phi_h]$ for any $h \in \Pi^{reg}$.
\end{definition}

Let $M_+(X)$ denote the subgroup of $M(X)$ consisting of all $f \in M(X)$ for which $sgn_+(f) = 1$. Then $M_+(X)$ has index $1$ or $2$ in $M(X)$. Observe that $\widehat{\mathbb{Z}}|_{M_+(X)} = \mathbb{Z}$ and thus $SW|_{M_+(X)} \in H^1( M(X) ; \mathbb{Z}[\mathcal{S}(X)])$. From $SW|_{M_+(X)}$ we can extract $\mathbb{Z}$-valued cohomology classes as follows: let $\mathcal{O} \subseteq \mathcal{S}(X)$ be a $\Gamma(X)$-invariant subset of $\mathcal{S}(X)$. Then we obtain a morphism $p_{\mathcal{O}} : \mathbb{Z}[\mathcal{S}(X)] \to \mathbb{Z}$ given by $\phi \mapsto \sum_{\mathfrak{s} \in \mathcal{O}} \phi(\mathfrak{s})$. We define $SW_{\mathcal{O}} \in H^1( M_+(X) ; \mathbb{Z})$ by setting $SW_{\mathcal{O}} = p_{\mathcal{O}}( SW|_{M_+(X)} )$. From this definition it follows that
\[
SW_{\mathcal{O}}|_{T(X)} = \sum_{\mathfrak{s} \in \mathcal{O}} SW_{\mathfrak{s}}.
\]

Furthermore, for any $f \in M_+(X)$, we have
\[
SW_{\mathcal{O}}(f) = \sum_{\mathfrak{s} \in \mathcal{O}} SW_{\mathfrak{s}}(h , f(h))
\]
where $h \in \Pi^{reg}$.

\begin{remark}
In \cite{rub2}, Ruberman defined an invariant $SW_{\mathfrak{s}}^{tot}$ which is similar to the definition of $SW_{\mathcal{O}}$ given above. Namely $SW_{\mathfrak{s}}^{tot} : M_+(X) \to \mathbb{Z}$ is given by $SW_{\mathfrak{s}}^{tot}(f) = \sum_{ \mathfrak{s}' } SW_{\mathfrak{s}'}( h , f(h) )$ where the sum is over all spin$^c$-structures $\mathfrak{s}'$ such that $\mathfrak{s}' = f^m(\mathfrak{s})$ for some $m \in \mathbb{Z}$. However this invariant is not a group homomorphism and behaves in a complicated manner with respect to composition (see \cite[Theorem 3.4]{rub2}). For this reason, we find it more useful to work with the invariants $SW_{\mathcal{O}}$.
\end{remark}

Let $\mathfrak{s} \in \mathcal{S}(X)$ be a spin$^c$-structure and let $f \in M(X)$. Suppose that $f$ preserves $\mathfrak{s}$. If $sgn_+(f) = 1$, then the families Seiberg---Witten moduli space for the mapping cylinder of $f$ with spin$^c$-structure $\mathfrak{s}$ is oriented, and we obtain an integer families Seiberg--Witten invariant $SW_{\mathfrak{s}}(f) \in \mathbb{Z}$. It is given by $SW_{\mathfrak{s}}(f) = SW_{\mathfrak{s}}( h , f(h) )$, for any $h \in \Pi^{reg}$. If $sgn_+(f) = -1$ then there is no natural choice of orientation on the families moduli space, hence we only get a mod $2$ invariant $SW_{\mathfrak{s}}(f) \in \mathbb{Z}_2$ which is given by $SW_{\mathfrak{s}}(f) = SW_{\mathfrak{s}}(h , f(h)) \; ({\rm mod} \; 2)$, for any $h \in \Pi^{reg}$ (the value of $SW_{\mathfrak{s}}(h,f(h))$ depends on $h$, but its mod $2$ reduction does not).

We will make use of the following special case of the gluing formula of \cite{bk1}:
\begin{proposition}\label{prop:glue}
Suppose that $X = X' \# (S^2 \times S^2)$, where $b_+(X') > 1$. Let $\mathfrak{s}'$ be a spin$^c$-structure on $X'$ with $d(\mathfrak{s}') = 0$ and let $\mathfrak{s}_0$ denote the spin structure on $S^2 \times S^2$. Let $\mathfrak{s} = \mathfrak{s}' \# \mathfrak{s}_0$. Let $f'$ be a diffeomorphism on $X'$ that preserves $\mathfrak{s}'$ and $\rho$ a diffeomorphism of $S^2 \times S^2$. Suppose that $f'$ is trivial in a neighbourhood of a point $x_1 \in X'$ and that $\rho$ is trivial in a neighbourhood of a point $x_2 \in S^2 \times S^2$. Set $f = f' \# \rho$, where the connected sum is performed by removing balls around $x_1$ and $x_2$ and identifying boundaries. Then:
\begin{itemize}
\item[(1)]{If $sgn_+(\rho) = 1$, then $SW_{\mathfrak{s}}(f) = 0 \; ({\rm mod} \; 2)$.}
\item[(2)]{If $sgn_+(\rho) = -1$, then $SW_{\mathfrak{s}}(f) = SW(X' , \mathfrak{s}') \; ({\rm mod} \; 2)$, where $SW(X' , \mathfrak{s}')$ denotes the ordinary Seiberg--Witten invariant of $(X' , \mathfrak{s}')$.}
\end{itemize}
\end{proposition}

%%%%%%%%%%%%%%%%%%%%%%%%%%%%%%%%%%%%%%%%%
\section{Non-finitely generated mapping class groups}\label{sec:sw2}

In this section we prove that $M(X)$ is not finitely generated for certain $4$-manifolds.

\begin{theorem}\label{thm:nfg}
Let $X = 2n \mathbb{CP}^2 \# 10n \overline{\mathbb{CP}^2}$, where $n \ge 3$ is odd. Then $M(X)$ is not finitely generated. More precisely, the following holds:
\begin{itemize}
\item[(1)]{There is a surjective homomorphism $\Phi : M_+(X) \to \mathbb{Z}^\infty$ from $M_+(X)$ to $\mathbb{Z}^{\infty}$, where $\mathbb{Z}^{\infty}$ denotes a free abelian group of countably infinite rank.}
\item[(2)]{The mod $2$ reduction of $\Phi$ extends to a homomorphism $\Phi : M(X) \to \mathbb{Z}_2^\infty$.}
\item[(3)]{$M_+(X)$ has index $2$ in $M(X)$.}
\item[(4)]{The short exact sequence $1 \to M_+(X) \to M(X) \to \mathbb{Z}_2 \to 0$ splits.}
\end{itemize}
\end{theorem}
\begin{proof}
First note that $X = X' \# (S^2 \times S^2)$, where $X' = (2n-1)\mathbb{CP}^2 \# (10n-1)\overline{\mathbb{CP}^2}$. It follows from \cite{wall} that $\Gamma(X) = Aut( H^2(X ; \mathbb{Z} ) )$. Observe that $d(\mathfrak{s}) = (c(\mathfrak{s})^2 + 8n)/4 - 2n - 1 = c(\mathfrak{s})^2/4 - 1$. Hence $d(\mathfrak{s}) = -1$ if and only if $c(\mathfrak{s})^2 = 0$. For each $k \ge 1$, let $\mathcal{O}_k \subset \mathcal{S}(X)$ denote the set of spin$^c$-structures whose characteristic $c$ satisfies, $c \neq 0$, $c^2 = 0$, and $c$ is $k$ times a primitive element. We will show that the homomorphism
\[
\Phi = \bigoplus_{q=1}^{\infty} SW_{\mathcal{O}_{nq-q-1}} : M_+(X) \to \mathbb{Z}^{\infty}
\]
surjects to a subgroup of $\mathbb{Z}^{\infty}$ of countably infinite rank. 

The decomposition $X = X' \# (S^2 \times S^2)$ yields an orthogonal decomposition $L = L' \oplus H$, where $L,L'$ are the intersection forms of $X,X'$ and $H = H^2(S^2 \times S^2 ; \mathbb{Z})$ is the hyperbolic lattice. Any characteristic $c \in L$ decomposes as $c = (c_1 , c_2)$, where $c_1 \in L', c_2 \in H$ are characteristic. The intersection form $L'$ is odd, hence $c_1 \neq 0$.

We will partition $\mathcal{O}_k$ into two types of subsets:
\begin{itemize}
\item[(1)]{Subsets $\{ \mathfrak{s} , \bar{\mathfrak{s}} \}$, where $\mathfrak{s} = \mathfrak{s}' \# \mathfrak{s}_0$ and $c(\mathfrak{s}_0) = 0$.}
\item[(2)]{Subsets $\{ \mathfrak{s}_1 , \mathfrak{s}_2 , \bar{\mathfrak{s}}_1 , \bar{\mathfrak{s}}_2 \}$, where $\mathfrak{s}_1 = \mathfrak{s}' \# \mathfrak{s}''$, $\mathfrak{s}_2 = \mathfrak{s}' \# \bar{\mathfrak{s}}''$ and where $c(\mathfrak{s}'') \neq 0$.}
\end{itemize}

Since $b_+(X) = 2n = 2 \; ({\rm mod} \; 4)$, we have $SW_{\mathfrak{s}}(t) = SW_{\bar{\mathfrak{s}}}(t)$ for all $t \in T(X)$. Hence a subset of type of $\mathcal{O}_k$ of type (1) will contribute $2 SW_{\mathfrak{s}}(t)$ to $SW_{\mathcal{O}_k}(t)$ and a subset of type (2) will contribute $2 ( SW_{\mathfrak{s}_1}(t) + SW_{\mathfrak{s}_2}(t) )$.

Let $E(n)_q$ be the elliptic surface obtained from $E(n)$ by performing a logarithmic transform of multiplicity $q \ge 1$. Since $n$ is odd, $E(n)_q$ is not spin and its intersection form is diagonal of signature $(2n-1 , 10n-1)$. Hence $E(n)_q$ has the same intersection lattice as $X'$. Furthermore, we have that $E(n)_q \# (S^2 \times S^2)$ is diffeomorphic to $2n \mathbb{CP}^2 \# 10n \overline{\mathbb{CP}^2} = X = X' \# (S^2 \times S^2)$ \cite[Page 320]{gs}. So there is an orientation preserving diffeomorphism $\psi_q : E(n)_q \# (S^2 \times S^2) \to X' \# (S^2 \times S^2)$. We can choose $\psi_q$ so that it respects the decomposition $H^2( E(n)_q ; \mathbb{Z}) \oplus H^2(S^2 \times S^2 ; \mathbb{Z}) \to H^2( X' ; \mathbb{Z}) \oplus H^2(S^2 \times S^2 ; \mathbb{Z})$. To see this, first let $\psi'_q : E(n)_q \# (S^2 \times S^2) \to X' \# (S^2 \times S^2)$ be any diffeomorphism. Then by \cite{wall}, every isometry of $H^2(X' \# (S^2 \times S^2) ; \mathbb{Z})$ can be realised by a diffeomorphism. Hence, composing $\psi'_q$ with a suitable diffeomorphism of $X$, we obtain the desired diffeomorphism $\psi_q$.

Let $\rho$ be a diffeomorphism of $S^2 \times S^2$ which acts as $-1$ on $H^2(S^2 \times S^2 ; \mathbb{Z})$ and is trivial in a neighbourhood of some point. Such diffeomorphisms can easily be constructed, for example, take the product $r \times r$ of two copies of a reflection on $S^2$ and then isotopy it to act trivially in a neighbourhood of a point. Define a diffeomorphism $f_0 \in M(X)$ by setting $f_ 0 = id_{X'} \# \rho$, where the connected sum is performed by removing a ball of $(S^2 \times S^2)$ on which $\rho$ acts trivially. For each $q \ge 1$, define a diffeomorphism $f_q \in M(X)$ by setting $f_q = \psi_q \circ ( id_{E(n)_q} \# \rho ) \circ \psi_q^{-1}$. Note that $sgn_+(f_0) = sgn_+(f_q) = -1$. Also $t_q = f_q \circ f_0 \in T(X)$.

We claim that $SW_{\mathcal{O}_{nq-q-1}}(t_q) = 2 \; ({\rm mod} \; 4)$ and that $SW_{\mathcal{O}_{nq'-q'-1}}(t_{q}) = 0 \; ({\rm mod} \; 4)$ for all $q' > q$. This implies that the elements $\{ \Phi(t_q) \}_{q \ge 1}$ are linearly independent. To see this, first note that $\Phi(t_q) \in 2\mathbb{Z}^{\infty}$ by charge conjugation symmetry and that the image of $\{ \Phi(t_q)/2 \}_{q \ge 1}$ is a basis for $\mathbb{Z}_2^\infty$, by the above claim. Now suppose $n_1 \Phi(t_1) + n_2 \Phi(t_2) + \cdots + n_r \Phi(t_r) = 0$ for some $n_1, \dots , n_r \in \mathbb{Z}$, not all zero. Without loss of generality we can assume that $gcd(n_1, \dots , n_r) = 1$. Then $n_1 \Phi(t_1)/2 + \cdots + n_r \Phi(t_r)/2 = 0$. But $\{ \Phi(t_q)/2\}_{q_ \ge 1}$ are linearly independent in $\mathbb{Z}_2^\infty$, so $n_1, \dots , n_r$ are all even, a contradiction.

Now we prove the claim. Let $q,q' \ge 1$ and set $k = nq'-q'-1$. By partitioning $\mathcal{O}_k$ into subsets of type (1) and (2) as described above, we can then write $SW_{\mathcal{O}_k}(t_{q})$ as a sum of contributions from sets of type (1) and (2). Consider a contribution from a subset $\{ \mathfrak{s} , \bar{\mathfrak{s}} \}$ of type (1). So $\mathfrak{s} = \mathfrak{s}' \# \mathfrak{s}_0$. The contribution is $2 SW_{\mathfrak{s}}( t_q )$. Since $f_q$ and $f_0$ both preserve $\mathfrak{s}$, we have that
\begin{align*}
SW_{\mathfrak{s}}(t_q) &= SW_{\mathfrak{s}}( f_q \circ f_0 ) \\
&= SW_{\mathfrak{s}}( f_q ) + SW_{\mathfrak{s}}( f_0) \; ({\rm mod} \; 2) \\
&= SW( E(n)_q , \mathfrak{s}' ) \; ({\rm mod} \; 2)
\end{align*}
where the last equality is due to Proposition \ref{prop:glue}. Let $f \in H^2( E(n)_q ; \mathbb{Z})$ denote the class of the multiple fibre. Then $f$ is non-zero, primitive and $f^2 = 0$. From the well-known calculation of the Seiberg--Witten invariants for elliptic surfaces \cite[Chapter 3]{nic}, we have that $SW( E(n)_q , \mathfrak{s}' ) = 0$ unless $c(\mathfrak{s}')$ is a multiple of $f$. More precisely, $SW(E(n)_q , \mathfrak{s}')$ is zero unless $c(\mathfrak{s}') = 2( qk + a)f - (nq-q-1)f$, where $0 \le k \le n-2$ and $0 \le a < q$. In such a case $SW(E(n)_q , \mathfrak{s}') = (-1)^k \binom{ n-2}{ k }$. In particular, $SW( E(n)_q , \mathfrak{s}') = \pm 1$ if $c(\mathfrak{s}') = (nq-q-1)f$ and $SW(E(n)_q , \mathfrak{s}') = 0$ if $c(\mathfrak{s}) = u f$, $u > nq-q-1$. Now suppose that $q' \ge q$. We have that $SW(E(n)_q , \mathfrak{s}') = 0$ unless $c(\mathfrak{s})$ is a multiple of $f$. But since $\mathfrak{s} = \mathfrak{s}' \# \mathfrak{s}_0 \in \mathcal{O}_k$, this can only happen if $c(\mathfrak{s}') = \pm k f$ (recall that $\mathcal{O}_k$ is the set of spin$^c$-structures whose characteristic $c$ satisfies, $c \neq 0$, $c^2 = 0$, and $c$ is $k$ times a primitive element). Hence if $q' > q$, then every pair $\{ \mathfrak{s} , \bar{\mathfrak{s}} \}$ of type (1) contributes $0$ mod $4$ to $SW_{\mathcal{O}_k}(t_q)$. If $q' = q$, then there is exactly one pair $\{ \mathfrak{s} , \bar{\mathfrak{s}} \}$ of type (1) which contributes $2$ mod $4$ to $SW_{\mathcal{O}_k}(t_q)$ and all other pairs are $0$ mod $4$.

Now consider the contribution from a set $\{ \mathfrak{s}_1 , \mathfrak{s}_2 , \bar{\mathfrak{s}}_1 , \bar{\mathfrak{s}}_2 \}$ of type (2), where $\mathfrak{s}_1 = \mathfrak{s}' \# \mathfrak{s}''$, $\mathfrak{s}_2 = \mathfrak{s}' \# \bar{\mathfrak{s}}''$ and where $c(\mathfrak{s}'') \neq 0$. As seen above, the contribution is $2 ( SW_{\mathfrak{s}_1}(t_q) + SW_{\mathfrak{s}_2}(t_q) )$. We will show that $SW_{\mathfrak{s}_1}(t_q) + SW_{\mathfrak{s}_2}(t_q) = 0 \; ({\rm mod} \; 2)$, hence all subsets of type (2) contribute $0$ mod $4$ and this will prove the claim.

Observe that $\mathfrak{s}_2 = f_q(\mathfrak{s}_1) = f_0(\mathfrak{s}_1)$. Let $h \in \Pi^{reg}$. Then
\begin{align*}
SW_{\mathfrak{s}_1}(t_q) &= SW_{\mathfrak{s}_1}( h , t_q(h) ) \\
&= SW_{\mathfrak{s}_1}(h , f_q f_0(h) ) \\
&= SW_{\mathfrak{s}_1}(h , f_q(h) ) + SW_{\mathfrak{s}_1}( f_q(h) , f_q f_0(h) ) \\
&= SW_{\mathfrak{s}_1}(h , f_q(h) ) - SW_{\mathfrak{s}_2}( h , f_0(h)).
\end{align*}
Similarly, $SW_{\mathfrak{s}_2}(t_q) = SW_{\mathfrak{s}_2}(h,f_q(h)) - SW_{\mathfrak{s}_1}(h,f_0(h))$. Hence
\begin{align*}
SW_{\mathfrak{s}_1}(t_q) + SW_{\mathfrak{s}_2}(t_q) &= \left( SW_{\mathfrak{s}_1}(h,f_q(h)) + SW_{\mathfrak{s}_2}(h,f_q(h)) \right) \\
& \quad \quad \quad - \left( SW_{\mathfrak{s}_1}(h,f_0(h)) + SW_{\mathfrak{s}_2}(h,f_0(h)) \right) \\
&= \left( SW_{\mathfrak{s}_1}(h,f_q(h)) - SW_{\mathfrak{s}_1}(f_q(h),f_q^2(h)) \right) \\ 
& \quad \quad \quad - \left( SW_{\mathfrak{s}_1}(h,f_0(h)) - SW_{\mathfrak{s}_1}(f_0(h),f_0^2(h)) \right) \\
&= SW_{\mathfrak{s}_1}( h , f_q^2(h) ) - SW_{\mathfrak{s}_1}( h , f_0^2(h) ) \\
&= SW_{\mathfrak{s}_1}( f_q^2 ) - SW_{\mathfrak{s}_1}( f_0^2 ) \\
&= 0 \; ({\rm mod} \; 2)
\end{align*}
where the last line follows from Proposition \ref{prop:glue}. This proves the claim. Hence we have proven that there exists a surjective homomorphism $M_+(X) \to \mathbb{Z}^\infty$.

The fact that the mod $2$ reduction of $\Phi$ extends to $M(X)$ follows by noting that $\frac{1}{2}SW_{\mathcal{O}_k} \in H^1( M_+(X) ; \mathbb{Z})$ extends to $H^1(M(X) ; \widehat{\mathbb{Z}})$ and then applying the mod $2$ reduction map $\widehat{\mathbb{Z}} \to \mathbb{Z}_2$.

The fact that $M_+(X)$ has index $2$ in $M(X)$ follows immediately from $sgn_+(f_0) = -1$. Furthermore, it is easy to see that $f_0^2$ is isotopic to a Dehn twist on the neck of the connected sum $X' \# (S^2 \times S^2)$. By using the circle action on $S^2 \times S^2$, it follows that this is isotopic to the identity. So $f_0$ defines a splitting $\mathbb{Z}_2 \to M(X)$ of the sequence $1 \to M_+(X) \to M(X) \to \mathbb{Z}_2 \to 0$.
\end{proof}

%%%%%%%%%%%%%%%%%%%%%%%%%%%%%%%%%%
\section{Split extensions}\label{sec:split}

Let $L = \mathbb{Z}^n$ denote the standard diagonal lattice of rank $n$ with orthonormal basis $e_1, \dots , e_n$. The isometry group of $L$ is the hyperoctahedral group $H_n$, which is also the Coxeter group of type $BC_n$. This group is easily seen to be generated by permutations of $e_1, \dots , e_n$ and the reflections $r_1, \dots , r_n$ in the hyperplanes orthogonal to $e_1, \dots , e_n$. The reflections generate a normal subgroup isomorphic to $\mathbb{Z}_2^n$ and $H_n$ is the semidirect product $H_n = S_n \ltimes \mathbb{Z}_2^n$.

Let $X = n \mathbb{CP}^2$ be the connected sum of $n$ copies of $\mathbb{CP}^2$. Then $H^2(X ; \mathbb{Z})$ is isomorphic to $L$ with $e_1, \dots , e_n$ corresponding to generators of $H^2(\mathbb{CP}^2 ; \mathbb{Z})$ for the $n$ summands of $X$. It is not hard to see that $\Gamma(X)$ is equal to the full isometry group of $L$. This will also follow from the construction given below.

\begin{theorem}\label{thm:split}
For $X = n \mathbb{CP}^2$, there is a splitting $\Gamma(X) \to M(X)$.
\end{theorem}
\begin{proof}
We will construct a smooth fibre bundle $\pi : E \to B$ with fibres diffeomorphic to $X$ and such that the monodromy of the local system $R^2 \pi_* \mathbb{Z}$ yields an isomorphism $\rho : \pi_1(B) \to Aut(L)$. The geometric monodromy of the family defines a lift $\widetilde{\rho} : \pi_1(B) \to \pi_0(Diff(X)) = M(X)$ of $\rho$ to $M(X)$. Then $\widetilde{\rho} \circ \rho^{-1} : Aut(L) \to M(X)$ is the desired splitting (this also proves that $\Gamma(X) \cong Aut(L)$).

Let $C_m$ be the space of $m$-tuples of distinct points on $S^4$. Clearly $C_1$ is diffeomorphic to $S^4$. For $m > 1$ there is a natural map $C_m \to C_{m-1}$ given by forgetting the $m$-th point. This map gives $C_m$ the structure of a fibre bundle over $C_{m-1}$ with fibre $F_{m-1}$ the $4$-sphere with $m-1$ points removed. Since $\pi_1(F_{m-1}) = \pi_0(F_{m-1}) = 1$, the long exact sequence in homotopy yields an isomorphism $\pi_1( C_{m}) \cong \pi_1(C_m)$. Since $\pi_1(C_1) = \pi_1(S^4) = 1$, it follows by induction that $\pi_1(C_n) = 1$ for all $n$.

Fix an orientation on $S^4$. Let $\widetilde{C}_n$ denote the space consisting of an $n$-tuple $(x_1 , \dots , x_n)$ of distinct points of $S^4$ together with an $n$-tuple $(I_1 , \dots , I_n)$, where $I_j$ is a complex structure on $T_{x_j} S^4$ which induces the given orientation. The forgetful map $\widetilde{C}_n \to C_n$ which forgets the complex structures $I_1, \dots , I_n$ gives $\widetilde{C}_n$ the structure of a fibre bundle over $C_n$. Since the space of complex structures on $\mathbb{R}^4$ compatible with a given orientation is isomorphic to $SO(4)/U(2) \cong S^2$, it follows that the fibres of $\widetilde{C}_n \to C_n$ are isomorphic to $(S^2)^n$. The long exact sequence in homotopy implies that $\pi_1( \widetilde{C}_n) = 1$.

Consider the trivial family $\widetilde{E}_0 = \widetilde{C}_n \times S^4 \to \widetilde{C}_n$. This family is equipped with $n$ sections $s_1, \dots , s_n$, where $s_j( (x_1 , \dots , x_n) , (I_1 , \dots , I_n)) = x_j$. The normal bundle of $s_j$ is $N_j = T_{x_j} S^4$. The complex structure $I_j$ gives $N_j$ the structure of a complex rank $2$ vector bundle. Therefore, we can form a family $\widetilde{E}_n$ by blowing up $\widetilde{E}_0$ along the sections $s_1, \dots , s_n$. More precisely, consider the fibre bundle over $\widetilde{C}_n$ with fibre $\mathbb{CP}^2$ given by the projective bundle $\mathbb{P}( \mathbb{C} \oplus N_j)$. This bundle has a natural section $t_j$ corresponding to the $1$-dimensional subbundle $\mathbb{C} \subset \mathbb{C} \oplus N_j$. The normal bundle of $t_j$ is isomorphic to $N_j$. Since $s_j$ and $t_j$ have isomorphic normal bundles, we can attach $\mathbb{P}(\mathbb{C} \oplus N_j)$ to $\widetilde{E}_0$ by removing tubular neighbourhoods of $s_j$ and $t_j$ and identifying the boundaries.

The hyperoctahedral group $H_n = S_n \ltimes \mathbb{Z}_2^n$ acts on $\widetilde{C}_n$ as follows. The permutation group $S_n$ acts by permuting the points $x_1, \dots , x_n$ as well as the corresponding complex structures $I_1, \dots , I_n$. The group $\mathbb{Z}_2^n$ is generated by reflections $r_1, \dots , r_n$. We let $r_j$ act by fixing $x_1, \dots , x_n$, sending $I_j$ to $-I_j$ and fixing the remaining complex structures. The action of $H_n$ is free and we let $B = \widetilde{C}_n/H_n$ be the quotient. It follows that $\pi_1(B) \cong H_n$. The action of $H_n$ on $\widetilde{C}_n$ lifts to an action on $\widetilde{E}_0 = \widetilde{C}_n \times S^4$ which acts trivially on the $S^4$ factor. It is not hard to see that $\widetilde{E}_n$ can be constructed in such a way that the action of $H_n$ extends to it. Now let $E = \widetilde{E}_n/ H_n$. This is a family $\pi : E \to B$ over $B$ with fibres diffeomorphic to $n \mathbb{CP}^2$. The monodromy of the local system $R^2 \pi_* \mathbb{Z}$ is easily seen to yield an isomorphism $\rho : \pi_1(B) \to Aut(L)$. As explained above, this yields a splitting $\Gamma(X) \to M(X)$.
\end{proof}

%%%%%%%%%%%%%%%%%%%%%%%%%%%%%%%%%%%%%
\section{Non-split extensions}\label{sec:nsplit}

Let $X$ be a compact, simply connected, smooth $4$-manifold with intersection form $L = H^2(X ; \mathbb{Z})$. Let $\Sigma$ be a compact surface (orientable or non-orientable). Suppose that $\rho : \pi_1(\Sigma) \to \Gamma(X)$ is a homomorphism. Letting $\Gamma(X)$ act on $L_{\mathbb{R}} = \mathbb{R} \otimes_{\mathbb{Z}} L$, we obtain a flat vector bundle $H_{\rho} \to \Sigma$ which has a covariantly constant bilinear form of signature $(b_+(X) , b_-(X))$. Let $H^+_{\rho}$ denote a maximal positive definite subbundle of $H_{\rho}$. The choice of subbundle $H^+_{\rho}$ is not unique, but all such subbundles are isomorphic. In particular the Stiefel--Whitney classes $w_{j}( H^+_{\rho}) \in H^j( \Sigma ; \mathbb{Z}_2)$ depend only on $\rho$.

\begin{theorem}\label{thm:nolift}
Let $X$ be a compact, simply connected, smooth $4$-manifold with $b_+(X)=2$ and let $\rho : \pi_1(\Sigma) \to \Gamma(X)$ be a homomorphism. Suppose that $w_2( H^+_{\rho}) \neq 0$ and suppose there exists a characteristic $c \in L$ which is $\rho$-invariant and satisfies $c^2 > \sigma(X)$. Then $\rho$ does not lift to a homomorphism $\widetilde{\rho} : \Gamma(X) \to M(X)$.
\end{theorem}
\begin{proof}
Consider first the case that $\Sigma$ is orientable of genus $g$. Recall that $\pi_1(\Sigma)$ admits a presention
\[
\pi_1(\Sigma) = \langle a_1,b_1, \dots , a_g,b_g \; | \; [a_1,b_1] \cdots [a_g,b_g] \rangle
\]
Suppose that $\rho$ admits a lift $\widetilde{\rho} : \pi_1(\Sigma) \to M(X)$. Let $\alpha_j$ be a diffeomorphism of $X$ whose isotopy class is $\widetilde{\rho}(a_j)$ and let $\beta_j$ be a diffeomorphism of $X$ whose isotopy class is $\widetilde{\rho}(b_j)$. Then $[\alpha_1,\beta_1] \cdots [\alpha_g,\beta_g]$ is isotopic to the identity. The surface $\Sigma$ can be constructed from a wedge of $2g$ circles by attaching a $2$-cell whose attaching map represents $[a_1,b_1] \cdots [a_g ,b_g]$ in $\pi_1( \vee_{i=1}^{2g} S^1)$. We will construct a smooth family $\pi : E \to \Sigma$ whose fibres are diffeomorphic to $X$ as follows. Over the $1$-skeleton $\vee_{i=1}^{2g} S^1$, we take the wedge sum of mapping cylinders associated to the diffeomorphisms $\alpha_1 , \beta_1 , \dots , \alpha_g , \beta_g$. A choice of isotopy from $[\alpha_1 , \beta_1] \cdots [\alpha_g , \beta_g]$ to the identity allows us to extend this family over the $2$-cell and in this way we obtain the family $\pi : E \to \Sigma$. By construction, the local system $R^2 \pi_* \mathbb{Z}$ has monodromy $\rho$. Now suppose that $w_2( H^+_{\rho}) \neq 0$ and that there exists a characteristic $c \in L$ which is $\rho$-invariant and satisfies $c^2 > \sigma(X)$. This contradicts \cite[Theorem 1.1]{bar1}, hence $\rho$ does not lift to $M(X)$.

The case that $\Sigma$ is non-orientable is similar. Recall that $\pi_1(\Sigma)$ admits a presentation
\[
\pi_1(\Sigma) = \langle a_1 , \dots , a_k \; | \; a_1^2 \cdots a_k^2 \rangle
\]
where $\Sigma$ has Euler characteristic $1-k$. If $\rho$ lifts to a homomorphism $\widetilde{\rho} : \pi_1(\Sigma) \to M(X)$, then we choose diffeomorphisms $\alpha_1 , \dots , \alpha_k$ where the isotopy class of $\alpha_j$ is $\widetilde{\rho}(a_j)$. Then $\alpha_1^2 \cdots \alpha_k^2$ is isotopic to the identity. A choice of such an isotopy allows us to construct a smooth family $\pi : E \to \Sigma$ with fibres diffeomorphic to $X$ and such that the monodromy of $R^2 \pi_* \mathbb{Z}$ is $\rho$. As before, this contradicts \cite[Theorem 1.1]{bar1}, hence $\rho$ does not lift to $M(X)$.
\end{proof} 

\begin{remark}
A similar argument was used in \cite{kmt} to prove the non-triviality of $T(X)$ for $X = 2\mathbb{CP} \# n \overline{\mathbb{CP}^2}$, $n \ge 11$.
\end{remark}

\begin{corollary}
Let $X = (S^2 \times S^2) \# X'$, where $b_+(X') = 1$, $b_-(X') \ge 10$. Then there does not exist a splitting $\Gamma(X) \to M(X)$.
\end{corollary}
\begin{proof}
Let $L = H^2(X ; \mathbb{Z})$, $L' = H^2(X' ; \mathbb{Z})$ denote the intersection lattices of $X$ and $X'$ and let $H = H^2(S^2 \times S^2 \; \mathbb{Z})$. So $L \cong H \oplus L'$. Since $X = (S^2 \times S^2) \# X'$, we have that $\Gamma(X) = Aut( H^2(X ; \mathbb{Z}))$ by \cite{wall}. Let $x,y \in H$ be a basis with $x^2 = y^2 = 0$, $\langle x , y \rangle = 1$. Since $b_+(X') = 1$ and $\sigma(X') < 0$, it follows that $X'$ is not spin and it follows that $L' \cong H' \oplus E_8 \oplus L''$, where $H'$ has basis $x',y'$, $(x')^2 = (y')^2 = 0$, $\langle x' , y' \rangle = 1$, $E_8$ is the negative definite $E_8$ lattice and $L''$ is a diagonal lattice with basis $e_1 , \dots , e_m$, where $m = b_-(X')-9$, with $e_i^2 = -1$ for all $i$, $\langle e_i , e_j \rangle = 0$ for $i \neq j$.

Let $u = x+y$, $v = x'+y'$. Then $u^2 = v^2 = 2$, $\langle u , v \rangle = 0$. Let $r_u$ be the reflection $r_u(x) = x + \langle x , u \rangle u$ and define $r_v$ similarly. Consider the isometry $f(x) = r_u r_v(x)$. Then $f \in \Gamma(X)$ and $f^2 = 1$. Hence we obtain a homomorphism $\rho : \pi_1(\mathbb{RP}^2) \to \Gamma(X)$ which sends the generator of $\pi_1(\mathbb{RP}^2)$ to $f$. Since $f$ acts as $-1$ on the maximal positive definite subspace of $H^2(X ; \mathbb{R})$ spanned by $u$ and $v$, we have that $w_2( H^+_{\rho}) \neq 0$. Let $c = e_1 +  \dots + e_m$. Then $c$ is a characteristic that $c^2 > \sigma(X)$ and $\langle c , u \rangle = \langle c , v \rangle = 0$. Then $r_u(c) = r_v(c) = c$ and hence $f(c) = c$. Then Theorem \ref{thm:nolift} implies that $\rho$ does not lift to $M(X)$. Hence the subgroup $\langle f \rangle \subseteq \Gamma(X)$ does not lift to $M(X)$, in particular, there does not exist a splitting $\Gamma(X) \to M(X)$.
\end{proof}

\begin{remark}
Note that in the above proof $u \in L$ can be realised by an embedded $2$-sphere in $X$, namely the diagonal $S^2 \subset S^2 \times S^2$. By a result of Seidel \cite{sei}, it follows that $r_u$ can be lifted to an element $\hat{r}_u \in M(X)$ of order $2$. Since $\Gamma(X) = Aut(L)$, it follows that there is a diffeomorphism of $X$ sending $u$ to $v$. It follows that $v$ can also be realised by an embedded $2$-sphere and hence $r_v$ can be lifted to an element $\hat{r}_v \in M(X)$ of order $2$. Then $\hat{r}_u \hat{r}_v$ is a lift of $f$ to $M(X)$. If $u,v$ could be represented by {\em disjoint} embedded $2$-spheres, then $\hat{r}_u, \hat{r}_v$ commute (since $\hat{r}_u, \hat{r}_v$ can be constructed to have disjoint supports) and then $\hat{r}_u \hat{r}_v$ would be an involutive lift of $f$, contradicting the corollary above. We deduce that $u,v$ can be represented by embedded spheres, but they can not be represented by disjoint embedded spheres even though $\langle u , v \rangle = 0$.
\end{remark}

%%%%%%%%%%%%%%%%%%%%%%%%%%%%%%%
\section{Nielsen realisation}\label{sec:nielsen}

As explained in the introduction, the following result shows that the Nielsen realisation problem fails for $X = X' \# p \mathbb{CP}^2 \# q \overline{\mathbb{CP}^2}$ whenever $p+q \ge 4$.

\begin{theorem}
Let $X = X' \# p \mathbb{CP}^2 \# q \overline{\mathbb{CP}^2}$ where $X'$ is a compact, smooth, simply-connected $4$-manifold and $p+q \ge 4$. Then $M(X)$ contains a subgroup isomorphic to $\mathbb{Z}_2^4$ which can not be lifted to $Diff(X)$.
\end{theorem}
\begin{proof}
To each summand of $\mathbb{CP}^2$ or $\overline{\mathbb{CP}^2}$ in $X$, there is a corresponding embedded $2$-sphere of self-intersection $\pm 1$. Let $E_1, \dots , E_4$ be any four of them. Let $t_1, \dots , t_4 \in M(X)$ be the corresponding Dehn twists around these spheres. Then $t_1,\dots , t_4$ are involutons \cite{sei} and they commute since $E_1, \dots , E_4$ are disjoint. Hence the group $G \subseteq M(X)$ generated by $t_1, \dots , t_4$ is isomorphic to $\mathbb{Z}_2^4$. Now suppose that $G$ can be lifted to $Diff(X)$. Hence we can find commuting diffeomorphisms $\sigma_1, \dots , \sigma_4$ such that the isotopy class of $\sigma_i$ is $t_i$.

Consider the fixed point set $F$ of $\sigma_1$. Since $\sigma_1$ acts on $H^2(X ; \mathbb{Z})$ as a reflection in a $\pm 1$ sphere, it follows from \cite[Proposition 2.4]{ed} that $F$ consists of a single copy of $\mathbb{RP}^2$, together with some isolated points and some $2$-spheres. Let $G_0$ be the subgroup of $G$ generated by $\sigma_2, \sigma_3, \sigma_4$. Since $\sigma_2, \sigma_3, \sigma_4$ commute with $\sigma_1$, they act on $F$ and in particular must send the copy of $\mathbb{RP}^2$ to itself. Hence $G_0$ acts on $\mathbb{RP}^2$. We claim that the action is effective. To see this, suppose $f \in G_0$ fixes $\mathbb{RP}^2$ pointwise. Since $f$ is an orientation preserving involution, it must act on the normal bundle of $\mathbb{RP}^2$ in $X$ as either the identity or multiplication by $-1$. Hence either $f$ or $\sigma_1 f$ fixes $\mathbb{RP}^2$ pointwise and acts trivially on the normal bundle. For a diffeomorphism of finite order, this can only happen if the diffeomorphism is the identity. Hence $f$ or $\sigma_1 f$ is the identity, but $f \in G_0$, so $f \neq \sigma_1$ and it must be that $f$ is the identity.

A finite group action on $\mathbb{RP}^2$ by diffeomorphisms is conjugate to a subgroup of $PO(3) \cong SO(3)$. Since $G_0$ is abelian, its action on the standard representation of $SO(3)$ can be simultaneously diagonalised, so $G_0$ is isomorphic to a subgroup of $\{ diag( \epsilon_1 , \epsilon_2 , \epsilon_3 ) \in SO(3) \} \cong \mathbb{Z}_2^2$, which is impossible since $|G_0| = 8$. So $G$ does not lift to $Diff(X)$.
\end{proof}

%%%%%%%%%%%%%%%%%%%%%%%%%%%%%%%%%%%%%%%%%%%%%%%%%%%%%%
\section{Boundary Dehn twists}\label{sec:bdt}

Let $X^{(n)}$ be obtained from $X$ by removing $n$ disjoint open balls. So $X^{(n)}$ is a compact $4$-manifold with boundary consisting of $n$ copies of $S^3$. Let $Diff(X^{(n)} , \partial X^{(n)})$ denote the group of diffeomorphisms of $X^{(n)}$ which are the identity in a neighbourhood of the boundary. Let $M_n(X) = \pi_0( Diff( X^{(n)} , \partial X^{(n)} )$ denote the group of components of $Diff( X^{(n)} , \partial X^{(n)})$. It is known that the map $M_n(X) \to M(X)$ is surjective and that the kernel is generated by Dehn twists on the boundary components \cite{gir}. More precisely, if $S^3 \to X^{(n)}$ is a boundary component, then $X^{(n)}$ has a tubular neighbourhood $[0,1] \times S^3 \to X$. The Dehn twist on this boundary component is defined by taking a non-trivial loop $\alpha_t :  [0,1] \to SO(4)$ and defining $\phi : [0,1] \times S^3 \to [0,1] \times S^3$ by $\phi(t,x) = (t , \alpha_t(x))$, where $SO(4)$ acts on $S^3$ in the standard way. We assume that $\alpha_t$ is smooth and equals the identity in a neighbourhood of $\{ 0,1\}$, hence $\phi$ can be extended to an element of $Diff(X^{(n)} , \partial X^{(n)})$ by taking it to be the identity outside of the tubular neighbourhood.

Let $K_n(X)$ denote the kernel of $M_n(X) \to M(X)$, so we have an short exact sequence
\[
1 \to K_n(X) \to M_n(X) \to M(X) \to 1.
\]
Furthermore, we have a surjection $\mathbb{Z}_2^n \to K_n(X)$ given by Dehn twists on the boundary components (\cite[Proposition 3.1]{gir}).

\begin{proposition}\label{prop:bdt}
Let $X$ be a compact, smooth, simply-connected $4$-manifold.
\begin{itemize}
\item[(1)]{If $X$ is spin, then $K_n(X)$ is either $\mathbb{Z}_2^n$ or $\mathbb{Z}_2^n/\Delta \mathbb{Z}_2$, for all $n$, where $\Delta \mathbb{Z}_2$ is the diagonal copy of $\mathbb{Z}_2$.}
\item[(2)]{If $X$ is not spin, then $K_n(X) = 0$ for all $n$, hence $M_n(X) \cong M(X)$.}
\end{itemize}
\end{proposition}
\begin{proof}
Part (1) is given by \cite[Corollary 2.5]{gir} and part (2) by \cite[Corollary A.5]{op}.
\end{proof}

In light of Proposition \ref{prop:bdt}, boundary Dehn twists are only interesting when $X$ is spin. In this case, we either have $K_n(X) \cong \mathbb{Z}_2^n$ or $K_n(X) \cong \mathbb{Z}_2^n/\Delta \mathbb{Z}_2$. Which of these two cases occurs is completely determined by the $n=1$ case. We consider this case in more detail. There is a Serre fibration 
\begin{equation}\label{equ:fib}
Diff( X^{(1)} , \partial X^{(1)} ) \to Diff(X) \to Emb\left( D^4 , X \right)
\end{equation}
where $Emb\left( D^4 , X \right)$ is the space of embeddings of a disc in $X$ which can be extended to a diffeomorphism. Furthermore, there is a homotopy equivalence $Emb\left( D^4 , X \right) \cong F(X)$, where $F(X)$ is the oriented frame bundle of $X$ \cite{gir}. Since $X$ is simply-connected and spin, $\pi_1(F(X)) \cong \mathbb{Z}_2$. Then the fibration (\ref{equ:fib}) induces an exact sequence
\[
\pi_1(Diff(X)) \buildrel \phi \over \longrightarrow \mathbb{Z}_2 \to M_1(X) \to M(X) \to 1.
\]
In the absence of a metric we can define the spin bundle of $X$ to be the universal cover $\widetilde{F}(X) \to F(X)$ of $F(X)$. Since $\pi_1(F(X)) \cong \mathbb{Z}_2$, $\widetilde{F}(X) \to F(X)$ is a double cover. Since $Emb\left( D^4 , X \right) \cong F(X)$, it follows that $\phi$ is the map that measures whether or not a loop of diffeomorphisms of $X$ lifts to a loop in the spin bundle of $X$. This leads to an alternative description of the group $M_1(X)$ when $X$ is spin. Let $SpinDiff(X)$ be the group whose elements consist of a diffeomorphism $f \in Diff(X)$ and a choice of lift of $f_* : F(X) \to F(X)$ to $\widetilde{F}(X)$. We have a short exact sequence $1 \to \mathbb{Z}_2 \to SpinDiff(X) \to Diff(X) \to 1$ and the connecting homomorphism $\pi_1(Diff(X)) \to \mathbb{Z}_2$ is precisely $\phi$. The map $Diff(X^{(1)} , \partial X^{(1)}) \to Diff(X)$ admits a lift $Diff(X^{(1)} , \partial X^{(1)} ) \to SpinDiff(X)$ by taking the unique lift which is the identity over $\partial X^{(1)}$. We then have a commutative diagram
\[
\xymatrix{
\pi_1(Diff(X)) \ar[r]^-{\phi} \ar[dr]_-{\phi} & \mathbb{Z}_2 \ar[r] \ar[d] & M_1(X) \ar[r] \ar[d] & M(X) \\
& \mathbb{Z}_2 \ar[r] & \pi_0(SpinDiff(X)) \ar[ur] & 
}
\]
from which it follows that $M_1(X) \to \pi_0( SpinDiff(X) )$ is an isomorphism. If $\phi$ is non-trivial, then $K_1(X) = 0$ and $M_1(X) \to M(X)$ is an isomorphism. This happens for $S^2 \times S^2$, as seen by taking a loop of diffeomorphisms given by a circle action which rotates one of the spheres. Similarly, $\phi$ is non-trivial for $X = S^4$ or for a connected sum of copies of $S^2 \times S^2$. If $\phi$ is trivial, then $K_1(X) \cong \mathbb{Z}_2$ and $M_1(X) \to M(X)$ is an extension of $M(X)$ by $\mathbb{Z}_2$, hence corresponds to a class $\xi_X \in H^2( M(X) ; \mathbb{Z}_2)$. It is natural to ask what this class is and in particular, whether or not it is trivial. First, we need some examples of spin $4$-manifolds where $\phi = 0$.

\begin{theorem}\label{thm:phi}
Let $X$ be a compact, smooth, simply-connected $4$-manifold. If $X$ is homeomorphic to $K3$ then $\phi = 0$. Similarly, if $X = X' \# (S^2 \times S^2)$, where $X'$ is homeomorphic to $K3$, then $\phi = 0$.
\end{theorem}
\begin{proof}
In \cite{bk2}, it is proven that if $E \to S^2$ is a smooth family of $K3$ surfaces over $S^2$, then $w_2(TE) = 0$. As explained in \cite{km}, this implies that the homomorphism $\phi$ is zero. The same argument works for any $X$ that is homeomorphic to $K3$, since by \cite{ms}, the Seiberg--Witten invariant of the spin structure of $X$ is odd.

Next, suppose $X = X' \# (S^2 \times S^2)$, where $X'$ is homeomorphic to $K3$. Suppose that $\phi$ is non-zero. This means that the boundary Dehn twist $\tau \in M_1(X)$ is trivial. But this would imply that the Dehn twist on the neck of $K3 \# X'$ becomes trivial upon connected sum with $S^2 \times S^2$. However this contradicts \cite{lin} (in \cite{lin} the theorem is stated only for $X' = K3$, but the exact same proof works for any smooth $4$-manifold homeomorphic to $K3$).
\end{proof}

Recall that an involution $f$ on a simply-connected spin $4$-manifold $X$ is called even or odd according to whether or not $f$ lifts to an involution on the spin bundle of $X$.

\begin{proposition}\label{prop:odd}
Suppose $X$ is spin and that $\phi = 0$, so that the extension class $\xi_X \in H^2(M(X) ; \mathbb{Z}_2)$ is defined. Suppose that $f$ is an odd involution. Then $\xi_X(f) \neq 0$. In particular, the extension $\mathbb{Z}_2 \to M_1(X) \to M(X)$ is non-trivial.
\end{proposition}
\begin{proof}
As explained above, the extension $1 \to \mathbb{Z}_2 \to M_1(X) \to M(X) \to 1$ is isomorphic to the extension $1 \to \mathbb{Z}_2 \to \pi_0(SpinDiff(X) ) \to M(X) \to 1$. But $f$ defines a class $[f] \in M(X)$ such that $[f]^2 = 1$, but any lift of $f$ to the spin bundle is not an involution. So there is no splitting $M(X) \to M_1(X)$ and more precisely, $\xi_X(f) \neq 0$.
\end{proof}

\begin{corollary}
If $X = K3$ or $K3 \# (S^2 \times S^2)$, then $\xi_X \in H^2( M(X) ; \mathbb{Z}_2)$ is non-trivial.
\end{corollary}
\begin{proof}
This is immediate from Theorem \ref{thm:phi} and Proposition \ref{prop:odd}, since both $K3$ and $K3 \# (S^2 \times S^2)$ admit odd involutions.
\end{proof}

In what follows we will completely determine the class $\xi_X \in H^2( M(X) ; \mathbb{Z}_2)$ when $X$ is homeomorphic to $K3$.

\begin{proposition}\label{prop:w2}
Let $\pi : E \to B$ be a smooth fibre bundle, where $B$ is a compact surface and the fibres of $E$ are diffeomorphic to a compact, simply-connected, smooth spin $4$-manifold $X$. Then
\begin{itemize}
\item[(1)]{There exists a spin$^c$-structure $\mathfrak{s}_{E/B}$ on the vertical tangent bundle $TE/B = Ker(\pi_*)$ whose restriction to each fibre is spin.}
\item[(2)]{Let $ind(D) \in K^0(B)$ denote the families index of the Dirac operator $D$ with respect to the spin$^c$-structure $\mathfrak{s}_{E/B}$. Then 
\[
c_1( ind(D) ) = (\sigma(X)/16) w_2( TE/B ) \; ({\rm mod} \; 2).
\]}
\end{itemize}
\end{proposition}
\begin{proof}
(1) follows immediately from \cite[Proposition 2.1]{bar1}. The Dirac operator $D$ for the spin$^c$-structure $\mathfrak{s}_{E/B}$ defines a family of elliptic operators parametrised by $B$ and $ind(D)$ is the families index. Then $c_1( ind(D)) = c_1(\mathcal{L})$, where $\mathcal{L} = det(ind(D))$ is the determinant line bundle of $D$. Suppose that the family $E$ is determined by transition function $\psi_{ij}$ valued in $Diff(X)$. Let $\widetilde{\psi}_{ij}$ be lifts of $\psi_{ij}$ to $SpinDiff(X)$. Then $\widetilde{\psi}_{ij} \widetilde{\psi}_{jk} \widetilde{\psi}_{ki} = g_{ijk}$, where $g_{ijk}$ is a $\mathbb{Z}_2$-valued cocycle, defining a class $[g_{ijk}] \in H^2( B ; \mathbb{Z}_2)$. Clearly $w_2( TE/B) = [g_{ijk}]$. Observe that $c(\mathfrak{s}_{E/B}) \in H^2( B ; \mathbb{Z})$ is a lift of $[g_{ijk}]$ to integer coefficients. Therefore we can represent $c(\mathfrak{s}_{E/B})$ as an integer-valued $2$-cocycle $c_{ijk}$ such that $c_{ijk} = g_{ijk} \; ({\rm mod} \; 2)$. Choose real-valued smooth functions $u_{ij}$ such that $c_{ijk} = u_{ij} + u_{jk} + u_{ki}$. Set $f_{ij} = e^{2 \pi i u_{ij}}$. Then $f_{ij}$ define transition functions for a complex line bundle whose first Chern class is $[c_{ijk}]$. Note that $f_{ij} = h_{ij}^2$, where $h_{ij} = e^{\pi i u_{ij}}$. Then $h_{ij} h_{jk} h_{ki} = (-1)^{g_{ijk}}$. Define $Spin^c Diff(X) = U(1) \times_{\mathbb{Z}_2} SpinDiff(X)$. Then $\varphi_{ij} = h_{ij} \widetilde{\psi}_{ij}$ is a $2$-cocycle valued in $Spin^cDiff(X)$.

Consider now the transition functions for the determinant line bundle $\mathcal{L}$. Since $\mathfrak{s}_{E/B}$ restricts to a spin structure on the fibres, the spinor bundles have a quaternionic structure on each fibre. It follows that $\widetilde{\psi}_{ij}$ induces a trivial action on the determinant line. However, the $U(1)$-factor $h_{ij}$ in $\varphi_{ij} = h_{ij} \widetilde{\psi}_{ij}$ acts on the spinor bundles as scalar multiplication which then acts on the determinant line by $h_{ij}^{d}$, where $d$ is the virtual rank of $ind(D)$, which is $d = -\sigma(X)/8$. Hence $\mathcal{L}$ has transition functions $h_{ij}^{-\sigma(X)/8} = f_{ij}^{-\sigma(X)/16}$. Recalling that $f_{ij}$ are transition functions for a line bundle with Chern class $c(\mathfrak{s}_{E/B})$, it follows that
\[
c_1(ind(D)) = c_1(\mathcal{L}) = -\frac{\sigma(X)}{16} c( \mathfrak{s}_{E/B}) = \frac{\sigma(X)}{16} w_2( TE/B) \; ({\rm mod} \; 2).
\]
\end{proof}

\begin{proposition}\label{prop:w2+}
Let $\pi : E \to B$ be a smooth fibre bundle, where the fibres of $E$ are homeomorphic to $K3$. Then $w_2(TE/B) = w_2(H^+)$, where $H^+ \to B$ denote the bundle whose fibre over $b$ is a maximal positive definite subspace of $H^2( E_b ; \mathbb{R})$.
\end{proposition}
\begin{proof}
Since $H^2(B ; \mathbb{Z}_2)$ is detected by maps of compact surfaces into $B$, it suffices to prove the result when $B$ is a compact surface. Then by Proposition \ref{prop:w2}, $w_2(TE/B) = c_1( ind(D)) \; ({\rm mod} \; 2)$. On the other hand, since the fibres are homeomorphic to $K3$, their Seiberg--Witten with respect to the spin structure is odd \cite{ms}. Then by \cite[Corollary 1.3]{bk2},  $c_1(ind(D)) = w_2(H^+)$.
\end{proof}

Let $L$ be a lattice and $A = Aut(L)$ the group of automorphisms. Over the classifying space $BA$ we have the tautological flat bundle $H = EA \times_{A} L$. Let $H^+ \to BA$ be a maximal positive subbundle. This defines a characteristic class $w_2(H^+) \in H^2( Aut(L) ; \mathbb{Z}_2)$.

\begin{theorem}
Let $X$ be a smooth $4$-manifold which is homeomorphic to $K3$. Let $L_X$ be the intersection lattice of $X$. Then the extension class $\xi_X \in H^2( M(X) ; \mathbb{Z}_2)$ is the pullback of $w_2(H^+) \in H^2( Aut(L_X) ; \mathbb{Z}_2 )$ under the map $M(X) \to Aut(L_X)$.
\end{theorem}
\begin{proof}
Let $B$ be a compact surface and consider a map $\iota : B \to BM(X)$. This is equivalent to a homomorphism $\rho : \pi_1(B) \to M(X)$. We claim that $\rho$ is the geometric monodromy of a family $E \to B$. We can take $B$ to be given by attaching a $2$-cell to a wedge of $k$ circles. Each circle defines a generator $g_i \in \pi_1(B)$ and the $2$-cell defines a relation $r = r(g_1, \dots , g_k)$, which is a word in the $g_i$. Choose a lift $f_i \in Diff(X)$ of $\rho(g_i) \in M(X)$. Then we can construct a family $E_1$ over the $1$-skeleton on $B$ as a wedge of mapping cylinders corresponding to the diffeomorphisms $f_1, \dots , f_k$. Since $g_1, \dots , g_k$ satisfy $r$, it follows that $r(f_1, \dots , f_k)$ is isotopic to the identity. Choosing such an isotopy, we can extend $E_1$ over the $2$-cell, giving the desired family $E \to B$. As explained in \cite[Remark 4.20]{bk2}, we can assume that the family $E \to B$ is smooth. Now consider the obstruction to lifting the structure group of $E$ to $SpinDiff(X)$. This is easily seen to coincide with the obstruction to lifting $\rho : \pi_1(B) \to M(X)$ to $M_1(X)$, which is $\iota^*( \xi_X) \in H^2(B ; \mathbb{Z}_2)$. On the other hand, the obstruction to lifting the structure group of $E$ to $SpinDiff(X)$ is $w_2(TE/B)$, which by Proposition \ref{prop:w2+} equals $w_2(H^+)$. Since $H^2( M(X) ; \mathbb{Z}_2)$ is detected by maps of compact surfaces $B$ into $BM(X)$, the result is proven.
\end{proof}

%%%%%%%%%%%%%%%%%%%%%%%%%%%%%%%%%%%%%%%%%%%%%%%%%%%%%%

\bibliographystyle{amsplain}

\end{document}